%% file: main.tex
\input{preamble}
\title{Strong Artin Conjecture for Generalized Octahedral Representations in \(\GL_3\)}

\author{Wang Junwu}

\begin{document}
	\maketitle
	
	\begin{abstract}
		We prove the strong Artin conjecture for the largest solvable Artin representations in three
		dimensions, whose projective image is isomorphic to the affine special linear group
		\(C_3^2\rtimes \SL(2, 3)\). It is also the last solvable case in three dimensions. This
		is achieved with a new case of base change and automorphic induction for non-Galois quartic
		extensions with no intermediate fields. We additionally deduce the strong Artin conjecture
		for all 6-dimensional primitive solvable Artin representations.
	\end{abstract}

	\tableofcontents
	\section{Introduction}

	Let \(F\) be a number field and let \(\Gamma_F \coloneq \Gal(\overline{F}/F)\) denote its
	absolute Galois group. Consider an irreducible Artin representation \(\rho: \Gamma_F \to
	\GL_n(\mathbb{C})\). The \textbf{strong Artin conjecture} states that there exists a cuspidal
	automorphic representation \(\pi\) of \(\GL_n(\mathbb{A}_F)\) such that for almost all places
	\(v\) of \(F\) where both \(\rho\) and \(\pi\) are unramified, the Satake parameters
	\(t_v(\pi)\) define the same semisimple conjugacy class of \(\GL_n(\mathbb{C})\) as
	\(\rho(\Frob_v)\). When this is the case, \(\pi\) is unique due to strong multiplicity one for
	\(\GL_n\). In particular, this implies that the Artin \(L\)-function \(L(s, \rho)\) is entire,
	which is the Artin conjecture.

	When \(n = 2\), the representation \(\rho\) is classified as dihedral, tetrahedral,
	octahedral or icosahedral according to its projective image in \(\PGL_2(\mathbb{C})\).
	The strong Artin conjecture for the first two cases was resolved by Langlands in
	\cite[Chapter 3]{Langlands} as an application of his work on base change for \(\GL_2\) for
	cyclic extensions of prime degree. Moreover, the proof of the tetrahedral case uses the adjoint
	lifting for \(\GL_2\) proven in \cite{GJ}. The octahedral case was resolved by Tunnell in
	\cite{Tunnell} by combining three ingredients, namely the tetrahedral case, base change for
	\(\GL_2\) for quadratic extensions as per \cite{Langlands}, as well as base change for \(\GL_2\)
	for non-Galois cubic extension as proven in \cite{JPSS-Cubic}. When \(F = \mathbb{Q}\) and
	\(\rho\) is odd icosahedral, the strong Artin conjecture was attacked through modularity lifting
	and \(p\)-adic deformation by Buzzard and Taylor in \cite{BT}, which requires residual 
	modularity. This lead to certain families of icosahedral representations being proven modular 
	by Buzzard, Dickinson, Shepherd-Barron and Taylor in \cite{BDST, Taylor}. When Serre's
	modularity conjecture was resolved by Khare and Wintenberger in \cite{KW}, the odd icosahedral
	case over \(\mathbb{Q}\) was then fully resolved.

	When \(n = 3\), a similar classification of \(\rho\) is known. If \(\rho\) is
	imprimitive, it is induced from a character from an index 3 subgroup of \(\Gamma_F\),
	which defines a cubic extension. By automorphic induction of Hecke characters for both
	Galois and non-Galois cubic extensions, due to Jacquet, Piatetski-Shapiro and Shalika
	\cite[Section 14.2]{JPSS-Cubic}, such representations \(\rho\) are automorphic. If \(\rho\) is
	primitive, its projective image in \(\PGL_3(\mathbb{C})\) is classified into six conjugacy
	classes, the isomorphism types of which was first detailed by Blichfeldt in \cite{Blichfeldt},
	namely
	\[
		H_{36}, H_{72}, H_{216}, A_5, \PSL(2, 7)\text{ and }A_6.
	\]
	The first three are solvable Hessian groups which are indexed by their orders.
	To define them, \(H_{216}\) is isomorphic to the affine special linear group
	\(C_3^2 \rtimes \SL(2, 3)\), where \(\SL(2, 3)\) acts by treating \(C_3^2\) as the vector space
	\(\mathbb{F}_3^2\). The groups \(H_{72}\) and \(H_{36}\) are subgroups of \(H_{216}\) where
	\(\SL(2, 3)\) is replaced by its 2-Sylow subgroup \(Q_8\), the quaternion group,
	or a subgroup isomorphic to \(C_4\). The last three cases in the list are the three smallest
	non-abelian simple groups.
	
	The strong Artin conjecture in the cases of \(H_{36}\) and \(H_{72}\) is similar to the
	tetrahedral case, in that they are resolved via base change and adjoint lifting
	for \(\GL_3\). Indeed, Lapid defined generalized tetrahedral representations for \(\GL_q\) where
	\(q\) a prime power in \cite[Section 5]{Lapid2}, and demonstrated that the adjoint lifting for
	\(\GL_q\) implies their modularity. His proof uses base change and automorphic induction
	for \(\GL_n\) over cyclic extensions of prime degree as shown by Arthur and Clozel in \cite{AC}.
	When \(q=3\), his definition exactly covers the cases \(H_{36}\) and \(H_{72}\).
	Thus, the cases \(H_{36}\) and \(H_{72}\) are automorphic as a consequence of the adjoint
	lifting of \(\GL_3\) due to Gan in \cite{Gan}. When the projective image of \(\rho\) is
	\(A_5\), we have that \(\rho\) is isomorphic to a quadratic twist of the adjoint lift of a
	two-dimensional icosahedral representation \(\sigma\), so the modularity of \(\rho\) is
	equivalent to the modularity of \(\sigma\) considering the Gelbart-Jacquet lift \cite{GJ}. To
	the best of our knowledge, no cases of \(\PSL(2, 7)\) and \(A_6\) are known to be automorphic.
	
	\subsection{Main result}

	In this paper, we prove the strong Artin conjecture for the case \(H_{216}\), which may be
	regarded as a generalization of octahedral representations. This resolves the strong Artin
	conjecture for all 3-dimensional Artin representations with solvable image.
		
	\begin{theorem}\label{main}
		Let \(\rho: \Gamma_F \to \GL_3(\mathbb{C})\) be a generalized octahedral representation, i.e.
		with projective image isomorphic to \(H_{216}\). Then \(\rho\) is automorphic.
	\end{theorem}

	The proof strategy and the rest of the paper is organized as follows. In section 2, we recall
	that certain functorial liftings and modularity results are strong, in the sense of respecting
	local Langlands correspondence at all places. In section 3, we prove a case of strong automorphic
	induction of Hecke characters for non-Galois quartic extensions using the
	exterior square lifting of \(\GL_4\) due to Kim in \cite{Kim}. This implies, via converse
	theorem, base change for \(\GL_2\) and \(\GL_3\) over such quartic extensions. We additionally
	give the full cuspidality criterion for both. Thereafter, in section 4, we prove the strong
	Artin conjecture for the case \(H_{216}\) by adapting the argument of Tunnell in \cite{Tunnell}.
	In the last section, we discuss primitive solvable representations in arbitrary dimensions
	using the work of Suprunenko in \cite{Suprunenko}. This motivates a definition of tetrahedral
	and octahedral representations in arbitrary dimensions, and allows us to deduce strong Artin
	conjecture for all primitive solvable Artin representations of dimension 6 using the
	Rankin-Selberg lift from \(\GL_2 \times \GL_3\) to \(\GL_6\) due to Kim and Shahidi in \cite{KS}.
	Throughout, we heavily use the existence, image and fibre of base change and automorphic
	induction for \(\GL_n\) over cyclic extensions of prime degree due to Arthur and Clozel in
	\cite{AC}.

	\subsection*{Acknowledgement}
	I thank Professor Gan Wee Teck for extensive advice and guidance at all stages of this work.

	\section{Preliminaries and notations}

	\subsection{Strong and weak correspondence}
	Let \(W_F\) denote the global Weil group of \(F\). For a place \(v\) of \(F\), let \(\rec_v\)
	denote local Langlands correspondence
	\[
		\rec_v: \Pi(\GL_n(F_v)) \to \Phi_n(F_v).
	\]
	When \(v\) is archimedean, the left hand side is the set of irreducible admissible 
	\((\mathfrak{g}, K)\)-modules and the right hand side is the set of continuous semisimple
	representations \(\phi: W_{F_v} \to \GL_n(\mathbb{C})\). This correspondence was established by
	Langlands in \cite{LanglandsLLC}. When \(v\) is non-archimedean, the left hand side is the set
	of irreducible admissible representations of \(\GL_n(F_v)\) and the right hand side is the set
	of Weil-Deligne representations \(\phi: W_{F_v} \times \SL_2(\mathbb{C}) \to
	\GL_n(\overline{\mathbb{Q}_\ell })\) for a prime number \(\ell\)
	not under \(v\). This correspondence was established by Harris and Taylor in \cite{HT}, and also
	by Henniart in \cite{HenniartLLC}. For simplicity, we denote by \(\mathcal{L}_{F_v}
	\coloneq W_{F_v}\) when \(v\) is archimedean and \(W_{F_v} \times \SL_2(\mathbb{C})\) when \(v\)
	is non-archimedean.

	For each place \(v\) of \(F\), fix an inclusion \(\overline{F}\hookrightarrow \overline{F_v}\),
	which determines an inclusion \(W_{F_v} \hookrightarrow W_F\). Let \(\rho: W_F \to 
	\GL_n(\mathbb{C})\) be a continuous semisimple representation.
	We say that an automorphic representation \(\pi\) of \(\GL_n(\mathbb{A}_F)\) 
	\textbf{strongly corresponds} to \(\rho\) if, for all places \(v\) of \(F\),
	\[
		\rec_v(\pi_v) \cong \rho|_{W_{F_v}}.
	\]
	This in particular means that \(\rec_v(\pi_v)\) is trivial on \(\SL_2(\mathbb{C})\) whenever
	\(v\) is non-archimedean. When this is the case, we denote \(\pi \xleftrightarrow{s} \rho\) and
	say that \(\rho\) is strongly automorphic.
	We say that \(\pi\) \textbf{weakly correponds} to \(\rho\) if the above is true for all except
	possibly a finite set of places \(S\), which customarily includes archimedean places and finite
	places where \(\rho\) is ramified. When this is the case, we denote \(\pi \leftrightarrow 
	\rho\) and say that \(\rho\) is automorphic.
	
	In the case of \(\GL_1\), we canonically identify characters of \(W_F\) and
	\(\GL_1(\mathbb{A}_F)\) via class field theory and shall not distinguish their
	spectral and automorphic natures, so that instead
	of \(\chi \xleftrightarrow{s} \chi'\) we shall simply write \(\chi = \chi'\).
	
	\subsection{Strongness of functorial liftings}\label{strong}

	We document that relevant functorial liftings and modularity results used in our proof are
	strong, in the sense of respecting local Langlands correspondence at all places.
	
	Arthur and Clozel \cite{AC} showed that base change and automorphic induction for \(\GL_n\) over
	cyclic extensions of prime degree respects local Langlands correspondence at archimedean and
	unramified finite places. At ramified finite places, because local Langlands correspondence was
	not yet proven, they showed that these functorial liftings satisfy certain character relations.
	In \cite[Lemma VII.2.6]{HT}, it is verified that they indeed respect a suitably normalized
	local Langlands correspondence. More precisely, let \(E/F\) be a cyclic extension of prime
	degree and let \(\BC\) and \(\AI\) denote base change and automorphic induction, respectively.
	Then one has, for all pairs of places \(w \mid v\) of \(E\) and \(F\) and a cuspidal
	representation \(\pi\) of \(\GL_n(\mathbb{A}_F)\),
	\[
		\rec_w(\BC_{E/F}(\pi)_w) \cong \rec_v(\pi_v)|_{\mathcal{L}_{E_w}}
	\]
	and dually for a cuspidal representation \(\pi'\) of \(\GL_n(\mathbb{A}_E)\),
	\[
		\rec_v(\AI_{E/F}(\pi')_v) \cong
		\bigoplus_{w\mid v} \Ind_{\mathcal{L}_{E_w}}^{\mathcal{L}_{F_v}}\rec_w(\pi'_w).
	\]
	We note that the strongness applies to automorphic induction and base change proved in
	\cite{JPSS-GL3, JPSS-Cubic} too, because they construct the global representation by defining
	the local component at every place which satisfies the required analytic properties to respect
	local Langlands correspondence.

	We make the following simple observation.
	
	\begin{proposition}\label{byhand}
		Let \(\rho: W_F\to \GL_n(\mathbb{C})\) be a semisimple continuous representation. For \(i=1,
		2,\ldots, m\), suppose there exist finite extensions \(E_i/F\) such that each may be refined
		to a tower of cyclic extensions of prime degree, and there exist Hecke characters
		\(\chi_i: W_{E_i}\to \mathbb{C}^\times\) such that
		\[
			\rho \cong \bigoplus_{i=1}^m \Ind_{W_{E_i}}^{W_F}(\chi_i).
		\]
		Then via class field theory and automorphic induction we define
		\[
			\pi \coloneq \bigboxplus_{i=1}^m \AI_{E_i/F}(\chi_i),
		\]
		which satisfies \(\pi \xleftrightarrow{s} \rho\).
	\end{proposition}
	
	Let us also recall the cuspidality criterion for base change, which is the counterpart of
	Clifford theory on the spectral side.

	\begin{proposition}[{\cite[Chapter 3, Theorem 4.2]{AC}}]\label{bc-cusp}
		Let \(E/F\) be a cyclic extension of number fields of prime degree \(\ell\). Let \(\sigma\) be
		a non-trivial element of \(\Gal(E/F)\), and let \(\omega_{E/F}\) be a non-trivial Hecke
		character associated with \(E/F\). Let \(\pi\) be a
		cuspidal representation of \(\GL_n(\mathbb{A}_F)\). Then \(\BC_{E/F}(\pi)\) is cuspidal
		except when \(\ell \mid n\) and \(\pi\cong\pi\otimes\omega_{E/F}\), in which case there
		exists a cuspidal representation \(\Pi\) of \(\GL_{n/\ell}(\mathbb{A}_E)\) such that
		\(\Pi \not\cong \Pi^\sigma\) and
		\[
			\BC_{E/F} \cong \bigboxplus_{i=0}^{\ell-1} \Pi^{\sigma^{i}}.
		\]
		Moreover, \(\pi \cong \AI_{E/F}(\Pi)\).
	\end{proposition}

	Kim in \cite{Kim} initially established that the exterior square lifting of \(\GL_4\) respects
	local Langlands correspondence except possibly for special cases at places dividing \(2\) or
	\(3\). Henniart later addressed these cases in \cite{HenniartGL4}. More precisely, let
	\(\wedge^2\) denote the exterior square lift and let \(\pi\) be a cuspidal representation of
	\(\GL_4(\mathbb{A}_F)\). Then one has
	\[
		\rec_v((\wedge^2\pi)_v) \cong \wedge^2 \rec_v(\pi_v).
	\]

	\subsection{Strongness of modularity results}
	
	The strong Artin conjecture admits an obvious refinement by extending to semisimple
	representations of the Weil group with finite projective image, and demanding strong 
	automorphy. In other words, if \(\rho: W_F \to \GL_n(\mathbb{C})\) is an irreducible semisimple
	continuous representation with finite projective image, then one conjectures that there exists a
	cuspidal representation \(\pi\) such that \(\pi\xleftrightarrow{s} \rho\). We note that for
	such \(\rho\), there exists an automorphic character \(\chi\) such that \(\rho \otimes \chi\)
	has finite image, i.e. factors through \(\Gamma_F\).

	Let \(\Gamma\) be a solvable finite subgroup of \(\PGL_n(\mathbb{C})\). The strong Artin
	conjecture for all Galois representations with projective image \(\Gamma\) is typicaly 
	proven using functorial liftings such as base change and adjoint lifting. These proofs typically
	extend verbatim to semisimple representations of the Weil group with projective image
	\(\Gamma\). Moreover, strong automorphy typically follows as a formal
	consequence of those functorial liftings being strong. We give an example of this principle by
	showing that two-dimensional semisimple representations of the Weil group with projective image
	\(A_4\) (i.e.\ tetrahedral) are strongly automorphic as a simple consequence of the strongness
	of base change and adjoint lifting for \(\GL_2\). We recall that the latter is strong because
	\cite{GJ} constructs the lift locally at every place with analytic properties matching the
	adjoint lift of local parameters. 

	\begin{proposition}\label{tetra}
		Let \(\rho: W_F \to \GL_2(\mathbb{C})\) be a semisimple irreducible continuous representation
		with projective image isomorphic to \(A_4\). Then there exists a cuspidal representation of
		\(\GL_2(\mathbb{A}_F)\) such that \(\pi \xleftrightarrow{s} \rho\).
	\end{proposition}

	\begin{proof}
		From Langlands' proof one can construct \(\pi\) such that the central character \(\omega_\pi\)
		matches \(\det(\rho)\), the base change of \(\pi\) to a cyclic cubic extension \(E/F\), denoted
		\(\pi_E\), satisfies \(\pi_E \xleftrightarrow{s} \rho|_{W_E}\), and the adjoint lifting
		satisfies \(\Ad \pi \xleftrightarrow{s} \Ad \rho\). Note that the
		second condition implies \(\rec_v(\pi_v)\) is trivial on the \(\SL_2(\mathbb{C})\) component
		when \(v\) is finite. For all places \(v\) of \(F\) completely split in \(E\), which includes
		archimedean places, the local correspondence is immediate. Otherwise, there is a unique place
		\(w\) of \(E\) lying over \(v\) and \(W_{E_w}\) is an index \(3\) subgroup of \(W_{F_v}\). We
		denote \(\phi_v \coloneq \rec_v(\pi_v)\). For any element \(g\in W_{F_v}\), we have
		the eigenvalues of \(\phi_v(g)\) match the eigenvalues of \(\rho(g)\) by product, cubes and 
		quotient, which allows one to match \(\tr(\phi_v) = \tr(\rho|_{W_{F_v}})\). Since the
		representations are semisimple and have equal trace, they are
		isomorphic and the proposition follows.
	\end{proof}

	\section{Automorphic induction and base change for \(A_4\)-quartic extensions}

	Let \(L/F\) be a Galois extension with \(\Gal(L/F) \cong A_4\). Let \(E/F\) be a subextension
	fixed by an order 3 element of \(A_4\). Then \(E/F\) is a non-Galois quartic
	extension with no intermediate subfields. Let \(\chi\) be a Hecke character of \(E\). In this
	section, we prove that the strong automorphic induction \(\AI_{E/F}(\chi)\) exists, which is
	defined as an automorphic representation satisfying \(\AI_{E/F}(\chi) \xleftrightarrow{s} 
	\Ind_{W_E}^{W_F}(\chi)\). The proof uses the exterior square lifting of \(\GL_4\) due to
	\cite{Kim} in an essential manner. This result then implies, via converse theorem, the existence
	of strong base change \(\BC_{E/F}(\pi)\) where \(\pi\) is a cuspidal representation of
	\(\GL_n(\mathbb{A}_F)\) for \(n = 2\text{ or }3\). This is defined as an automorphic
	representation respecting the local Langlands correspondence at all places.
	
	We first set up the field extensions considered.
	Let \(K/F\) be a subextension of \(L/F\) fixed by \(V_4\), the Klein four group, in \(A_4\).
	There are three intermediate quadratic extensions \(M_i/K\) for \(i = 1,2,3\), which are fixed
	by nontrivial elements \(a, b, c \in V_4\), respectively. The following diagram illustrates the
	field extensions, where solid lines denote Galois extensions and dotted lines denote
	non-Galois extensions, and labels denote Galois groups.

	
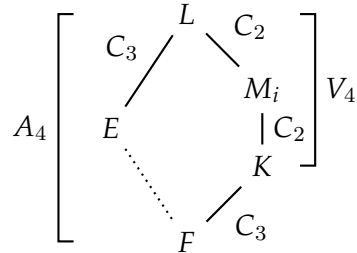
\begin{figure}[htbp]
		\centering
		\begin{tikzpicture}[thick]
			\node (L) at (0, 3) {\(L\)};
			\node (M) at (1, 2) {\(M_i\)};
			\node (K) at (1, 1) {\(K\)};
			\node (E) at (-1, 1.5) {\(E\)};
			\node (F) at (0, 0) {\(F\)};

			\draw (L) -- (M) node[midway, above right] {\(C_2\)};
			\draw (M) -- (K) node[midway, right] {\(C_2\)};
			\draw (L) -- (E) node[midway, above left] {\(C_3\)};
			\draw (K) -- (F) node[midway, below right] {\(C_3\)};
			
			\draw[dotted] (E) -- (F);
			
			\draw (-1.5, 3) -- (-1.7, 3) -- node[left] {\(A_4\)} (-1.7, 0) -- (-1.5, 0);
			
			\draw (1.5, 3) -- (1.7, 3) -- node[right] {\(V_4\)} (1.7, 1) -- (1.5, 1);
		\end{tikzpicture}
	\end{figure}
	
	We shall often write the restriction of a representation of the Weil group to a finite index
	subroup simply by subscript of the target field. For example, \(\rho|_{W_K}\) is simply written
	as \(\rho_K\). Dually, the base change of an automorphic representation to a finite extension
	is also written with subscript of the field, so that \(\BC_{K/F}(\pi)\) is written as \(\pi_K\).
	
	In the process of proving strong automorphic induction and base change, we shall also give the
	full cuspidality criterion, which characterises the degenerate cases in which \(\AI_{E/F}\) and
	\(\BC_{E/F}\) are non-cuspidal. For this purpose, the following lemma is useful.

	\begin{lemma}\label{degen}
		Let \(n=2\) or \(3\). Let \(\rho: W_F \to \GL_n(\mathbb{C})\) be an irreducible semisimple
		continous representation. Then the following are equivalent:
		\begin{enumerate}
			\item The restriction \(\rho_L\) is an isotypic sum of Hecke characters, i.e.
				\(\rho_L \cong \chi \oplus\chi\) or \(\rho_L \cong \chi\oplus\chi\oplus \chi\).
			\item The projective image \(\Im(\overline{\rho})\cong A_4\) which exactly cuts out \(L/F\).
		\end{enumerate}
		Moreover, \(\rho\) is strongly automorphic, so that there exists a cuspidal \(\pi\) such that
		\(\pi \xleftrightarrow{s} \rho\).

		Dually, let \(n=2\) or \(3\) and let \(\pi\) be a cuspidal automorphic representation of
		\(\GL_n(\mathbb{A}_F)\). Then the following are equivalent:
		\begin{enumerate}
			\item The base change \(\pi_L\) is an isotypic isobaric sum of Hecke characters, i.e. \(\pi
				\cong\chi\boxplus\chi\) or \(\pi \cong\chi\boxplus\chi\boxplus\chi\).
			\item There exists \(\rho\) of the previous form such that \(\pi\xleftrightarrow{s} \rho\). 
		\end{enumerate}
	\end{lemma}

	\begin{proof}
		In both cases, (2) \(\Rightarrow\) (1) is clear. We shall first prove (1) \(\Rightarrow\) (2)
		for \(\rho\). The kernel of the projectivization \(\ker(\overline{\rho})\) contains \(W_L\),
		so \(\Im(\overline{\rho})\) is a quotient of \(A_4\). The only quotients of \(A_4\) are
		\(\{e\},C_3\) and \(A_4\). If \(\Im(\overline{\rho})\not\cong A_4\), then \(\Im(\rho)\) is the
		extension of a cyclic group by a central subgroup, which is abelian. Since \(\rho\) is
		irreducible and not 1-dimensional, we must have \(\ker(\overline{\rho}) = W_L\). It remains
		to show that \(\rho\) is strongly automorphic. When \(n=2\), this is Proposition \ref{tetra}.
		When \(n=3\), because \(A_4\) is not a primitive subgroup of \(\PGL_3(\mathbb{C})\), we have
		\(\rho\) is monomial and thus strongly automorphic due to automorphic induction through any
		cubic extension by \cite{JPSS-GL3}.
		
		We now show (1)\(\Rightarrow\) (2) for \(\pi\). For this purpose, we consider the chain of 
		base changes through cyclic extensions of prime degree, i.e. \(\pi_K\), \(\pi_{M_1}\) and
		finally \(\pi_L\). When \(n=2\), by consulting Proposition \ref{bc-cusp} we must have 
		\(\pi_K\) is cuspidal and \(\pi_{M_1}\) is non-cuspidal. In particular, \(\pi_K\) is induced
		from a Hecke character \(\theta\) of \(M_1\), so we can construct \(\rho_K \coloneq
		\Ind_{W_{M_1}}^{W_K}(\theta)\). By construction, \(\rho_K\) is stable under \(\Gal(K/F)\), so
		that it has three extensions \(\rho_1, \rho_2, \rho_3\) to \(W_K\), all of which have
		projective image cutting out \(L/F\) . Thus they are strong automorphic and give rise to
		distinct cuspidal \(\pi_1, \pi_2, \pi_3\) which base change to \(\pi_K\). By the fiber of 
		global base change, exactly one of \(\pi_i\) must be \(\pi\), which proves the lemma for
		\(n=2\). For \(n=3\), another consultation of Proposition \ref{bc-cusp} shows that \(\pi_K\)
		must be an isobaric sum of Hecke characters and \(\pi\) is the automorphic induction of a
		Hecke character \(\theta'\) of \(K\). Hence we can construct
		\(\rho\coloneq \Ind_{W_K}^{W_F}(\theta')\) which is of the required form.
	\end{proof}

	We are now ready for the main theorems on functorial liftings.

	\begin{proposition}\label{ai}
		Let \(L, E, F\) and \(\chi\) be as above and write \(\sigma \coloneq \Ind_{W_E}^{W_F}(\chi)\).
		There exists an automorphic representation \(\tau\) of \(\GL_4(\mathbb{A}_F)\) such that
		\(\tau \xleftrightarrow{s} \sigma\).
	\end{proposition}

	\begin{remark}
		This case of automorphic induction is a generalization of Martin's work on the modularity of
		certain four-dimensional Artin representations called hypertetrahedral representations in
		\cite{Martin}, because those representations are monomial and may be written as
		\(\Ind_{W_E}^{W_F}(\chi)\). Our work is a generalization because \(\chi\) must satisfy
		specific properties for \(\Ind_{W_E}^{W_F}(\chi)\) to be hypertetrahedral. For example,
		\(\chi^2_L\) must be  invariant under \(\Gal(L/F)\). Nevertheless, our strategy is
		broadly similar to that of Martin in the use of exterior square lifting for \(\GL_4\).
	\end{remark}

	\begin{proof}
		We first perform computations on the spectral side. By Mackey, we have 
		\(\sigma_K \cong \Ind_{W_L}^{W_K}(\chi_L)\). Moreover, because the centralizer of \(\chi_L\)
		in \(A_4\) contains \(\Gal(L/E)\) and \(A_4\) has no subgroups of order 6, we have that 
		either \(\chi_L\) is fixed by \(V_4\) or \(\chi_L\) has four distinct conjugates under \(V_4\).
		In the former case, \(\sigma\) is reducible and we will prove the
		proposition by constructing \(\tau\) ``by hand''. In the latter
		case, \(\sigma\) is irreducible and we will need the additional exterior square lift.
		
		\textbf{Case \(\chi_L\) invariant}

		The character \(\chi_L\) defines a cohomological class \([\alpha]\in H^2(V_4,
		\mathbb{C}^\times) \cong C_2\) which is its Mackey obstruction class.
		When \([\alpha]\) is trivial, there are four extensions of \(\chi_L\) to \(W_K\), and
		\(\Gal(K/F) \cong C_3\) permutes three of them while fixing the last. Let \(\theta\) be the
		character of \(W_K\) extending \(\chi_L\) and fixed by \(\Gal(K/F)\). Then the other three
		extensions may be written as \(\theta\cdot\omega_i\) for \(i=1,2,3\), where \(\omega_i\) is the
		quadratic character associated with the quadratic extension \(M_i/K\). In this 
		case, \(\sigma_K\) is the direct sum of these four characters, and \(\sigma\) decomposes as
		\(\Theta \oplus \Lambda\) of dimensions 1 and 3 respectively. Specifically, \(\Theta\) is the
		character defined by \(\Theta(nh) = \theta(n)\chi(h)\) where \(n\in W_K\) and \(h\in W_E\),
		and \(\Lambda\) is the representation \(\Ind_{W_K}^{W_F}(\theta \cdot \omega_i)\) for any
		\(i = 1, 2, 3\). Therefore, by Proposition \ref{byhand}, \(\tau \coloneq \Theta \boxplus
		\tau_\Lambda\) satisfies \(\tau \xleftrightarrow{s} \sigma\), where \(\tau_\Lambda\) is
		cuspidal of \(\GL_3\) and strongly corresponds to \(\Lambda\).
		It is clear that \(\Lambda\) and \(\tau_\Lambda\) are of the form in Lemma \ref{degen}.

		When \([\alpha]\) is the nontrivial class, the data \((\theta, [\alpha])\) uniquely defines a
		two-dimensional irreducible representation \(\Psi\) of \(W_K\) with the property that
		\(\Psi_{M_i}\) decomposes as the direct sum of the two extensions of \(\chi_L\) to \(M_i\).
		Let \(\omega_{L/E}\) be a nontrivial character associated with the cyclic cubic extension
		\(L/E\). By Mackey, we have
		\[
			\Ind_{W_K}^{W_F}(\Psi)|_{W_E} \cong
			\Ind_{W_L}^{W_E}(\Psi|_{W_L}) \cong
			2 \cdot \bigoplus_{i=0}^2 \chi \otimes \omega_{L/E}^i.
		\]
		In other words, it is the sum of two copies each of the three twists of \(\chi\).
		Therefore, via Frobenius reciprocity,
		\[
			\dim \Hom_{W_F}\left(\sigma, \Ind_{W_K}^{W_F}(\Psi)\right) =
			\dim \Hom_{W_E}\left(\chi, \Ind_{W_K}^{W_F}(\Psi)|_{W_E}\right) = 2.
		\]
		Since \(\sigma\) is 4-dimensional and \(\Ind_{W_K}^{W_F}(\Psi)\) is 6-dimensional, this shows
		that \(\Ind_{W_K}^{W_F}(\Psi) \cong \Psi_1 \oplus \Psi_2 \oplus \Psi_3\), where
		\(\Psi_j\) is a representation restricting to \(\Psi\) over \(W_K\).
		It is clear that \(\Psi_j\) are of the form in Lemma \ref{degen}.
		Now, \(\Psi_j|_{W_E}\) is the sum of two characters, which cannot be identical, as otherwise
		the kernel of the projectivization \(\ker(\overline{\Psi_j})\) is strictly larger than
		\(W_L\) and contradicts Lemma \ref{degen}. Hence, by Frobenius reciprocity one sees that
		there are exactly two \(\Psi_j\) whose restriction to \(W_E\) contains
		\(\chi\), which are \(\Psi_1\) and \(\Psi_2\) without loss of generality, and \(\sigma
		\cong\Psi_1 \oplus \Psi_2\). Hence by Lemma \ref{degen}, we have
		\(\tau_{\Psi, j} \xleftrightarrow{s} \Psi_j\) and \(\tau\coloneq \tau_{\Psi, 1}
		\boxplus \tau_{\Psi, 2}\) satisfies \(\tau \xleftrightarrow{s} \sigma\).
		
		\textbf{Case \(\chi_L\) not invariant}
		
		In this case, \(\sigma_K\) is irreducible.
		Consider the cuspidal automorphic representation of \(\GL_4(\mathbb{A}_K)\)
		defined as \(\tau_K \coloneq \AI_{L/K}\left( \chi_L \right)\).
		Let \(x\in \Gal(K/F)\) be nontrivial, and let \(\tilde{x} \in \Gal(L/F)\) be its unique lift
		which resides in \(\Gal(L/E)\). We compute
		\[
			\tau_K^x \cong \AI_{L/K}(\chi_L^{\tilde{x}})) \cong \AI_{L/K}(\chi_L) = \tau_K, 
		\]
		because \(\chi_L\) is invariant under \(\Gal(L/E)\). Therefore, \(\tau_K\) is invariant under
		\(\Gal(K/F)\). By characterisation of the image and fibre of base change, there exists a
		cuspidal automorphic representation \(\tau\) of \(\GL_4(\mathbb{A}_F)\) such that
		\(\BC_{K/F}(\tau) = \tau_K\). Moreover, the fibre of this base change is precisely
		\(\left\{\tau \otimes \omega_{K/F}^j\right\}\), where \(\omega_{K/F}\) a non-trivial Hecke 
		character associated with \(K/F\) and \(j = 0, 1, 2\).

		We now identify a single candidate that can correspond to \(\sigma\) via central
		character. Let \(\omega_\tau\) be the central character of \(\tau\).
		For \(j = 0, 1, 2\), the central character of \(\tau \otimes \omega_{K/F}^j\) is
		\(\omega_\tau \omega_{K/F}^{4j} = \omega_\tau \omega_{K/F}^j\). Consequently, the three 
		candidates have distinct central characters.	
		The expected central character is \(\det(\sigma)\). Because the central
		character of \(\tau_K\) is \(\det(\sigma)|_{W_K}\), we have that
		\(\omega_\tau = \det(\sigma) \omega_{K/F}^k\) for some
		\(k \in \{0, 1, 2\}\). Hence, there is a unique \(\tau\) satisfying \(\BC_{K/F}(\tau) \cong
		\tau_K\) and \(\omega_\tau = \det(\sigma)\).
		
		We now wish to show that that \(\tau\xleftrightarrow{s}\sigma\). To this end, we first show
		that
		\begin{lemma}
		 The exterior square lifting \(\wedge^2 \pi\) due to Kim in \cite{Kim} satisfies
		 \(\wedge^2 \pi \xleftrightarrow{s}\wedge^2 \rho\).
		\end{lemma}
		\begin{proof}
			We first show that \(\wedge^2 \rho\), a 6-dimensional representation of \(W_F\), is strongly
			automorphic by constructing the automorphic representation ``by hand''. This shall follow
			from some more computation on the spectral side. Let \(P\) denote the set of two element
			subsets of \(V_4\). By Mackey, we compute
			\[
				(\wedge^2\sigma)|_{W_L} \cong \bigoplus_{\{x, y\}\in P} \chi_L^x\cdot \chi_L^y.
			\]

			If \(\chi_L\cdot\chi_L^a\) is invariant under \(V_4\) with trivial Mackey
			obstruction, then \((\wedge^2 \sigma)|_{W_K}\) decomposes as the direct
			sum of six characters, and the action of \(\Gal(K/F)\) arranges them into two orbits of
			three characters each. Consequently, \(\wedge^2 \sigma\) is the sum of two inductions of
			characters from \(W_K\) and we may construct \(\Pi' \xleftrightarrow{s} \wedge^2\sigma\) via
			Proposition \ref{byhand}. In this case, \(\Pi' \cong \Pi'_1 \boxplus \Pi'_2\) where
			each cuspidal summand is of \(\GL_3\).

			Otherwise, either \(\chi_L\cdot\chi_L^a\) is not invariant under \(V_4\) or it is invariant
			under \(V_4\) with non-trivial Mackey obstruction. In both cases,
			\((\wedge^2 \sigma)|_{W_K}\)
			decomposes as \(\rho_1 \oplus \rho_2 \oplus \rho_3\) where each irreducible component is two
			dimensional and induced from a Hecke character of \(M_i\). For example, to define
			\(\rho_1\), observe that \(\chi_L\cdot\chi_L^a\) is invariant under \(a\). Let \(\theta\)
			be one of the two extensions of \(\chi_L\cdot\chi_L^a\) to \(W_{M_1}\). Then \(\rho_1\) is
			the induction of \(\theta\) to \(W_K\). Now because \(\Gal(K/F)\) permutes the three
			\(\rho_i\), we have that \(\wedge^2 \sigma\) is irreducible and in particular can be
			obtained as \(\Ind_{W_{M_1}}^{W_F}(\theta)\). Hence, by Proposition \ref{byhand}, we may
			construct a cuspidal automorphic representation \(\Pi'\) of \(\GL_6\), which thus satisfies
			\(\Pi' \leftrightarrow \wedge^2 \sigma\).

			Returning to the automorphic side, \(\tau\) admits an exterior square lift \(\Pi\) of
			\(\GL_6(\mathbb{A}_F)\) due to \cite{Kim}, which is known to be isobaric.
			It remains to show that \(\Pi \cong \Pi'\).
			Let \(w\) be a place of \(K\) over an unramified place \(v\) of \(F\) such that both \(\Pi\)
			and \(\Pi'\) are unramified. We compute that
			\[
				t_w(\Pi_K) = \wedge^2 t_w(\tau_K) = \wedge^2 \sigma_K(\Frob_w) =
				t_w(\Pi'_K).
			\]
			Therefore, by strong multiplicity one for isobaric automorphic representations, we have that
			\(\Pi_K \cong \Pi'_K\) as isobaric automorphic representations. The
			characterization of the fibre of cyclic base change then dictates that \(\Pi\) and \(\Pi'\)
			has the same shape of isobaric decomposition, and there exists a bijection of cuspidal
			summands where each pair differs by twisting by a power of \(\omega_{K/F}\).
			However, each cuspidal summand of \(\Pi'\) is obtained as an automorphic induction from
			\(K\), which means it is stable under twisting by \(\omega_{K/F}\). Therefore, we conclude
			that \(\Pi \cong \Pi'\), and thus \(\wedge^2 \tau \xleftrightarrow{s} \wedge^2 \rho\).
		\end{proof}
		
		We finally verify that \(\tau \xleftrightarrow{s} \rho\).
		Let \(v\) be a place of \(F\) and \(w\) be a place of \(K\) lying over \(v\).
		If \(v\) is finite, then because
		\[
			\rec_v(\tau_v)|_{W_{K_w}\times \SL_2(\mathbb{C})} \cong \rec_w(\tau_{K, w})
			\cong \sigma|_{W_{K_w}}
		\]
		is trivial on \(\SL_2(\mathbb{C})\), we have that \(\rec_v(\tau_v)\) is also trivial on
		\(\SL_2(\mathbb{C})\). Thus we will treat \(\rec_v(\tau_v)\) as a representation of \(W_{F_v}\)
		regardless if \(v\) is archimedean or non-archimedean.

		Let us now write \(\phi_v \coloneq \rec_v(\tau_v)\) and \(\sigma_v \coloneq
		\sigma|_{W_{F_v}}\). Since they are semisimple, it suffices to show that \(\tr(\phi_v) =
		\tr(\sigma_v)\). Consider an element \(g\in W_{F_v}\). If \(g\in W_{K_w}\), then due to
		strong base change, we have \(\tr(\phi_v(g)) = \tr(\sigma_v(g))\). Otherwise
		\(g\in W_{F_v}\backslash W_{K_w}\), which happens only when \(v\) is finite inert or totally
		ramified. In this case, \(\sigma_v(g)\) permutes three lines
		in \(\mathbb{C}^4\) while fixing the last. This implies that its multiset of eigenvalues of
		\(\sigma_v(g)\), denoted \(S\), can be written as \(S = \{a, b, b\omega, b\omega^2\}\), where
		\(a, b \in \mathbb{C}^\times\) and \(\omega\) is a primitive cube root of unity.

		On the other hand, the multiset of eigenvalues of \(\phi_v(g)\), denoted \(T\),
		match \(S\) in three ways:
		
		\begin{enumerate}
			\item The multiset of cubes match, denoted \(S^3 = T^3\), because \(g^3 \in W_{K_w}\) and
				\(\tau_K \xleftrightarrow{s} \sigma_K\).
			\item The product of elements in each multiset match, denoted \(\prod S = \prod T\), because
				\(\omega_\tau = \det(\sigma)\).
			\item The multisets of pairwise products match, denoted \(\wedge^2 S = \wedge^2 T\), because
				\(\wedge^2 \tau \xleftrightarrow{s} \wedge^2 \rho\).
		\end{enumerate}
			
		The claim now is that these conditions force \(S = T\). This is a surprisingly
		delicate algebraic exercise, which we formulate as the next lemma.
		This implies that \(\tr(\phi_v) = \tr(\sigma_v)\) on all of \(W_{F_v}\), which implies that
		they are isomorphic, proving the proposition.
	\end{proof}
	
	\begin{lemma}
		Let \(S, T\) be two multisets of four non-zero complex numbers. Suppose
		\(S = \{a, b, b\omega, b\omega^2\}\) where \(a, b \in \mathbb{C}^\times\) and \(\omega\) is a
		primitive cube root of unity. If \(S\) and \(T\) satisfy
		\begin{enumerate}
			\item The multiset of cubes match, denoted \(S^3 = T^3\).
			\item The product of elements in each multiset match, denoted \(\prod S = \prod T\). 
			\item The multisets of pairwise products match, denoted \(\wedge^2 S = \wedge^2 T\).
		\end{enumerate}
		Then \(S = T\). 
	\end{lemma}

	\begin{remark}
		If \(\mathbb{C}^\times\) is replaced by a torsion-free abelian group, then it is a classical
		result of Selfridge and Straus in \cite{SS} that the multiset of pairwise products determine a
		multiset of size \(n\) if and only if \(n \neq 2^m\), \(m \ge 0\). Moreover, they characterise
		all counterexamples when \(n = 2^m\). In our case, the three numbered conditions alone are
		insufficient to rule out all counterexamples, which makes the explicit description of \(S\)
		necessary. For example, the multisets \(\{1, 1, 1, \omega^2\}\) and
		\(\{\omega, \omega, \omega, \omega^2\}\) satisfy all three numbered conditions without being
		identical.
	\end{remark}
	
	\begin{proof}
		From the first condition, we see that \(T = \{a\omega^x, b\omega^y, b\omega^z, b\omega^w\}\). 
		The shared product of elements is \(ab^3\). There is a unique way to partition the shared
		multiset of pairwise products \(U\) into three unordered pairs with each pair having product
		\(ab^3\), namely \(U_1 = \{ab, b^2\}, U_2 = \{ab\omega, b^2\omega^2\}, U_3 = 
		\{ab\omega^2, b^2\omega\}\).
		
		Suppose \(\abs{a} \neq \abs{b}\). Because we can identify those elements of \(U\) which
		involve \(a\omega^x\), up to permuting \(y,z,w\) one can deduce that
		\begin{align*}
			x + y &\equiv z + w\\
			x + z &\equiv y+w - 1\\
			x + w &\equiv y+z+1\\
			x + y + z + w &\equiv 0
		\end{align*}
		with all equalities in \(\mathbb{Z}/3\mathbb{Z}\). This has a unique solution \((x, y, z, w)=
		(0, 0, 1, 2)\), which implies \(S = T\).

		Now suppose that \(a = b e^{ci}\) for \(c \in [0, 2\pi)\). If \(c\) is not a multiple of
		\(2\pi/3\), we can again deduce those elements of \(U\) involving \(a\) in the pairwise
		product and arrive at the same system of equations above, in which case \(S = T\).

		Thus we are left with the case where up to twisting by a power of \(\omega\) and
		dividing by \(b\), we have \(S = \{1, 1, \omega, \omega^2\}\) and \(T\) is composed of cubic
		roots of unity \(\{\omega^x, \omega^y, \omega^z, \omega^w\}\). There must be pair whose
		sum is \(0\), so suppose without loss of generality that \(x+y \equiv 0\).
		Thus the product condition forces \(z+w\equiv 0\), too. There must be another pair whose
		sum is \(1\), so suppose without loss of generality that \(x+z\equiv 1\). Then this defines a
		system of three linear equations which has three solutions
		\((x, y, z, w) = (0, 0, 1, 2), (1, 2, 0, 0)\) or \((2, 1, 2, 1)\). The last solution does not
		satisfy the pairwise product condition and the first two imply \(S = T\).
	\end{proof}
	
	The proof of the proposition contains the cuspidality criterion of \(\tau\), which we state as a
	corollary.

	\begin{corollary}
		In the context of Proposition \ref{ai}, \(\tau\coloneq \AI_{E/F}(\chi)\) is an isobaric sum of
		cuspidal automorphic representations. Specifically,
    \begin{enumerate}
      \item When \(\chi_L\) is not invariant under \(\Gal(L/K)\), \(\tau\) is cuspidal.
      \item When \(\chi_L\) is invariant and its Mackey obstruction class \([\alpha]\) is trivial,
      \(\tau \cong \Theta \boxplus \tau_\Lambda\), of dimensions 1 and 3.
      \item When \(\chi_L\) is invariant and \([\alpha]\) is nontrivial,
      \(\tau \cong \tau_{\Psi, 1} \boxplus \tau_{\Psi, 2}\), of dimensions 2 and 2.
    \end{enumerate}
		Moreover, in the non-cuspidal case, the base change of each cuspidal summand to \(L\) is an
		isotypic isobaric sum of copies of \(\chi_L\).
	\end{corollary}

	For the purpose of proving strong base change, we state the exact version of converse theorem we
	are using, which is adapted from \cite[Theorem 11.3]{JL} and \cite[Theorem 13.6]{JPSS-GL3}. The
	more general version for \(\GL_n\) was proved by Cogdell and Piatetski-Shapiro in \cite{CPS}.
	
	\begin{theorem}[Converse Theorem]
		Let \(\pi\) be an irreducible admissible representation of \(\GL_n(\mathbb{A}_F)\),
		where \(n = 2, 3\), whose central character \(\omega_\pi\) is trivial on \(F^\times\) and whose
		L-function \(L(\pi, s)\) is convergent in some right half plane. Suppose 
		\(L(\pi \times \chi, s)\) is nice (i.e. admits analytic continuation to the entire complex
		plane, is bounded in vertical strips, and satisfies an appropriate functional equation) for
		every Hecke character \(\chi\). Then \(\pi\) is cuspidal automorphic.
	\end{theorem}

	\begin{proposition}\label{bc}
		Let \(L, E, F\) be as above. For \(n = 2\text{ or }3\) and a cuspidal
		automorphic representation \(\pi\) of \(\GL_n(\mathbb{A}_F)\), there
		exists an automorphic representation \(\Pi\) of \(\GL_n(\mathbb{A}_E)\) which is the
		strong base change of \(\pi\).
  \end{proposition}
  
  \begin{proof}
		As with the proof of automorphic induction, certain degenerate cases shall be computed ``by
		hand'' by going to the spectral side.

		\textbf{Degenerate case}
		
		We say that \(\pi\) is degenerate if it is of the form in Lemma \ref{degen}.
		Note that \(\Im(\rho_E)\) is the extension of a cyclic group \(\Gal(L/E) \cong C_3\)
		by a central subgroup, i.e. \(\Im(\rho_L)\). Hence \(\Im(\rho_E)\) is abelian and \(\rho_E\)
		decomposes completely as a direct sum of linear characters. We may construct the required
		\(\Pi\) as the isobaric sum of these characters, which thus satisfies the condition of strong
		base change.

		\textbf{Generic case}

		Suppose \(\pi\) is not of the form in Lemma \ref{degen}. For each place \(w\) of \(E\) over a
		place \(v\) of \(F\), we define the local component \(\Pi_w\) via local Langlands
		correspondence, i.e.
		\[
			\Pi_w \coloneq \rec_w^{-1}\left(\rec_v(\pi_v)|_{\mathcal{L}_{E_w}}\right).
		\]
		We then form the global representation \(\Pi = \bigotimes_w \Pi_w\) of 
		\(\GL_n(\mathbb{A}_E)\). It is clear that \(\Pi\) has an automorphic central character and its
		\(L\)-function, defined as an Euler product, is convergent in some right half plane. 

    Let \(\chi\) be a Hecke character of \(E\). By Proposition \ref{ai}, the strong automorphic
		induction \(\tau \coloneq \AI_{E/F}(\chi)\) exists. We thus compute, for the real part of 
		\(s\) sufficiently large,
    \[
			L(s, \Pi \times \chi) = \prod_w L(s, \Pi_w \times \chi_w) = \prod_v L(s, \pi_v
			\times \tau_v) = L(s, \pi \times \tau),
    \]
    where for the middle equality, we pass to the spectral side via local Langlands correspondence
		and invoke Frobenius reciprocity. Thus, via uniqueness of analytic continuation, we have that
		\(L(s, \Pi\times \chi) = L(s, \pi \times \tau)\).
		
    Let \(\tau \cong \boxtimes_j \tau_j\) be the isobaric decomposition of \(\tau\) where each
		\(\tau_j\) is cuspidal automorphic. We can analyze this \(L\)-function through its
		factorization
    \[
			L(s, \pi \times \tau) = \prod_j L(s, \pi \times \tau_j).
		\]
		Recall that the Rankin-Selberg \(L\)-function \(L(s, \pi \times \tau_j)\) is nice
		except possibly for a simple pole at \(s=1\). Moreover, it has such a pole precisely
		when \(\pi\) is isomorphic to the contragredient \(\tau_j^\vee\). Now, in the previous
		corollary, whenever the components \(\tau_j\) have dimensions 2 or 3, they are of the
		degenerate case. It is clear then that \(\tau_j^\vee\) are also of the form in the degenerate
		case. Since \(\pi\) is not of this form, it is not isomorphic to any such
		\(\tau_j^\vee\), so that \(L(s, \pi \times \tau)\) is nice. Now, admitting analytic
		continuation and being bounded in vertical strips are simply properties of analytic functions,
		so \(L(s, \Pi \times \chi)\) satisfies these. As for functional equation, the exact form
		required is such that
		\[
			L(s, \Pi \times \chi) = \epsilon(s, \Pi\times \chi) L(1-s, \Pi^\vee \times \chi^{-1}).
		\]
		By invoking the functional equation of \(L(s, \pi\times\tau)\), we need only check that
		\(\epsilon(s, \Pi\times \chi) = \epsilon(s, \pi\times \tau)\). Because \(\epsilon\)-factors 
		are also defined as a product of local factors, we can similarly invoke local Langlands
		correspondence followed by Frobenius reciprocity. Fixing an additive character \(\psi:
		F\backslash \mathbb{A}_F \to \mathbb{C}\), this yields
		\[
			\epsilon(s, \pi_v\times\tau_v, \psi_v) = \prod_{w\mid v}
			\lambda(E_w/F_v, \psi_v)^n\epsilon(s, \Pi_w\times\chi_w, \psi_w)
		\]
		Here, the \(\lambda\)-factors are as defined by Langlands. Because the product of
		\(\lambda\)-factors is equal to \(1\), we have \(\epsilon(s, \pi\times\tau) =
		\epsilon(s, \Pi\times\chi)\). Consequently, \(L(s, \Pi \times
		\chi)\) satisfies the required functional equation and therefore is nice. Therefore, by
		converse theorem, \(\Pi\) is cuspidal automorphic.
  \end{proof}
	
	We similarly state the cuspidality criterion obtained in the proof.

	\begin{corollary}
		In the context of Proposition \ref{bc}, \(\Pi \coloneq \BC_{E/F}(\pi)\) is cuspidal except
		when \(\BC_{L/F}(\pi)\) is an isotypic isobaric sum of Hecke characters, in which case \(\Pi\)
		is an isobaric sum of Hecke characters. The degenerate case may happen for both \(n=2\) and
		\(n=3\).
	\end{corollary}

	\section{Automorphy of generalized octahedral representations}
	
	In this section, we establish the modularity of generalized octahedral
	representations in three dimensions. We adapt the method used for the octahedral case in
	\cite{Tunnell}, so it is helpful to sketch Tunnell's argument.

	Suppose \(\rho: \Gamma_F \to \GL_2(\mathbb{C})\) is octahedral, i.e. has projective image
	\(S_4\). Let \(L_\text{proj}/F\) be the \(S_4\)-Galois extension so defined. The normal subgroup
	\(V_4 \subset S_4\) defines an \(S_3\)-Galois extension \(L/F\). Let \(K/F\) be the
	subextension fixed by \(C_3 \subset S_3\) and let \(E/F\) be fixed by an order 2 element of
	\(S_3\), which is non-Galois cubic. The following diagram illustrates the fields considered.

	
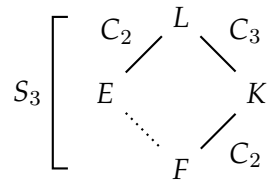
\begin{figure}[htbp]
		\centering
		\begin{tikzpicture}[thick]
			\node (L) at (0, 2) {\(L\)};
			\node (K) at (1, 1) {\(K\)};
			\node (E) at (-1, 1) {\(E\)};
			\node (F) at (0, 0) {\(F\)};

			\draw (L) -- (K) node[midway, above right] {\(C_3\)};
			\draw (L) -- (E) node[midway, above left] {\(C_2\)};
			\draw (K) -- (F) node[midway, below right] {\(C_2\)};
			
			\draw[dotted] (E) -- (F);
			
			\draw (-1.5, 2) -- (-1.7, 2) -- node[left] {\(S_3\)} (-1.7, 0) -- (-1.5, 0);
		\end{tikzpicture}
	\end{figure}

	One notices that the restriction \(\rho\) to \(W_K\) is tetrahedral, which
	is known to be modular. However, when one invokes base change descent to \(F\), one cannot
	distinguish the two candidates through central character, because the degree of extension is
	not coprime to the dimension of the representation. Tunnell observes that the restriction of
	\(\rho\) to \(W_E\) is dihedral, which is also known to be modular. Moreover, the automorphic
	analogue, i.e. non-normal cubic base change for \(\GL_2\), was established by
	\cite{JPSS-Cubic}. These two known modular cases then impose enough algebraic conditions to
	isolate a single candidate and prove modularity of \(\rho\).

	We use the same framework, because after defining a similar set of field extensions,
	a 3-dimensional generalized octahedral representation \(\rho\) becomes generalized tetrahedral
	upon restriction to \(W_K\), and becomes monomial upon restriction to an index 4 non-normal
	subgroup \(W_E\) which defines an \(A_4\)-quartic extension, the base change through which is
	established in the previous section. 
	
	\begin{proof}[Proof of Theorem \ref{main}]
		Recall that \(\rho: \Gamma_F \to \GL_3(\mathbb{C})\) has projective image \(C_3^2\rtimes
		\SL(2, 3)\). Let \(L_{\text{proj}}/F\) be the Galois extension so defined.
		This group contains a normal subgroup \(N \cong (C_3)^2 \rtimes C_2\), where
		\(C_2\) is the center of \(\SL(2,3)\), with quotient group \(A_4\).
		Let \(L\) be the subfield fixed by \(N\), so that \(L/F\) is a Galois
		extension with \(\Gal(L/F) \cong A_4\). Inside \(\Gal(L/F)\), let \(V_4\) denote the normal
		Klein four-group. We define \(K\) to be the subfield fixed by \(V_4\), and \(E\) to be the
		subfield fixed by an order \(3\) element of \(A_4\). Thus the
		extension \(E/F\) is an \(A_4\)-quartic extension with Galois closure \(L\), which places us
		in the setting of the previous section. For easy reference, we reproduce a simpler version of
		the diagram of fields in the previous section.
		
		
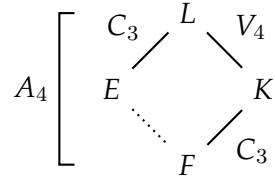
\begin{figure}[htbp]
			\centering
			\begin{tikzpicture}[thick]
				\node (L) at (0, 2) {\(L\)};
				\node (K) at (1, 1) {\(K\)};
				\node (E) at (-1, 1) {\(E\)};
				\node (F) at (0, 0) {\(F\)};

				\draw (L) -- (K) node[midway, above right] {\(V_4\)};
				\draw (L) -- (E) node[midway, above left] {\(C_3\)};
				\draw (K) -- (F) node[midway, below right] {\(C_3\)};
				
				\draw[dotted] (E) -- (F);
				
				\draw (-1.5, 2) -- (-1.7, 2) -- node[left] {\(A_4\)} (-1.7, 0) -- (-1.5, 0);
			\end{tikzpicture}
		\end{figure}

		The restrictions of \(\rho\) to the intermediate
		fields \(K\) and \(E\) are automorphic. Indeed, the restriction \(\rho_K\) has projective
		image corresponding to \(H_{72}\), so that it is generalized tetrahedral and is automorphic as
		a corollary of adjoint lifting of \(\GL_3\) by Gan in \cite{Gan}.
		Let \(\pi_K\) be the corresponding cuspidal representation on
		\(\GL_3(\mathbb{A}_K)\) such that \(\pi_K \leftrightarrow \rho_K\). The restriction \(\rho_E\)
		is imprimitive as its projective image is \(C_3^2 \rtimes C_6\) which is not one of the
		primitive groups. In fact, it is induced from a non-Galois cubic extension of \(E\), and is
		thus also automorphic by non-Galois cubic automorphic induction of Hecke characters by
		\cite{JPSS-GL3}. Let \(\pi_E\) be its corresponding cuspidal representation on
		\(\GL_3(\mathbb{A}_E)\) such that \(\pi_E \xleftrightarrow{s} \rho_E\).

		For each \(g\in \Gal(K/F)\), because \(\pi_K^g \leftrightarrow \rho_K^g\) and \(\rho_K^g \cong
		\rho_K\), we have \(\pi_K\) is invariant under \(\Gal(K/F)\). Hence, there exists exactly 
		three cuspidal representations \(\pi_1, \pi_2, \pi_3\) on \(\GL_3(\mathbb{A}_F)\) such that
		\(\BC_{K/F}(\pi_i) = \pi_K\). Moreover, they are twists of each other by a nontrivial Hecke
		character associated with \(K/F\), denoted \(\omega_{K/F}\).
		One now wishes to fix one out of the three candidate descents.
		
		Let \(\pi_L\) denote \(\BC_{L/E}(\pi_E)\), so that \(\pi_L \xleftrightarrow{s} \rho_L\).
		By Proposition \ref{bc}, \(\pi_{i, E} \coloneq \BC_{E/F}(\pi_i)\)
		exists and is cuspidal. Because successive restrictions of local Langlands parameters commute,
		we also have that \(\BC_{L/E}(\pi_{i, E}) \cong \pi_L\). Because \(L/E\) is a cyclic cubic
		extension, the fibre of the base change \(\BC_{L/E}\) consists of twists by a Hecke character
		\(\omega_{L/E}\) associated with \(L/E\). Therefore, \(\pi_E \cong \pi_{1, E} \otimes
		\omega_{L/E}^j\) for some \(j \in \{0, 1, 2\}\).
    
    The base change of \(\pi_1 \otimes \omega_{K/F}^i\) to \(E\) is 
		\(\pi_{1, E} \otimes (\omega_{K/F} \circ N_{E/F})^i\).
		Because \(K \cap E = F\) and \(KE = L\), the composition \(\omega_{K/F} \circ N_{E/F}\) is a
		nontrivial character associated with \(L/E\), which we assume to be \(\omega_{L/E}\) without
		loss of generality. Thus, the base changes of the three descents to \(E\) are exactly
		\(\pi_{1, E} \otimes \omega_{L/E}^i\). Because \(\pi_E\) is, exactly
		one of these twists, it follows that exactly one of \(\pi_i\) base changes to
		\(\pi_E\) over \(E\). We fix \(\pi\) to be this unique descent.
   
		It remains to show that \(\pi \leftrightarrow \rho\).
		Let \(v\) be a finite place of \(F\) unramified for \(L/F\), \(\pi\) and \(\rho\). The
		Frobenius conjugacy class of \(v\) in \(\Gal(L/F) \cong A_4\) must have order 1, 2, or 3.

    If \(\Frob_v\) has order 1 or 2 in \(\Gal(L/F)\), its image in the quotient
    \(\Gal(K/F) \cong C_3\) is the identity, meaning \(v\) splits completely in the cyclic cubic
    extension \(K/F\). Thus, there exists a place \(u\) of \(K\) dividing \(v\) with local
    degree 1, so \(K_u = F_v\). Because we know \(\pi_K\leftrightarrow\rho_K\), we have
		\(t_u(\pi_K) = \rho_K(\Frob_u)\). Via the identity local base change from \(F_v\) to \(K_u\),
		we deduce that \(t_v(\pi) = \rho(\Frob_v)\). Otherwise, if \(\Frob_v\) has order 3 in
		\(\Gal(L/F)\), it acts as a 3-cycle on the four cosets of \(C_3\) in \(A_4\). This permutation
		leaves exactly one coset fixed, meaning \(v\) splits in the quartic extension \(E/F\) into a
		place \(w\) of degree 1 and another of degree 3. Because we fixed \(\pi\) such that
		\(\pi_E \xleftrightarrow{s} \rho_E\), we deduce that \(t_w(\pi_E) = \rho_E(\Frob_w)\). Via
		the identity local base change from \(F_v\) to \(E_w\), this again guarantees that
		\(t_v(\pi) = \rho(\Frob_v)\). Thus, almost all local components
		of \(\pi\) and \(\rho\) agree, proving that \(\pi \leftrightarrow \rho\).
	\end{proof}
	
	\begin{remark}
		It is natural to ask whether \(\pi \xleftrightarrow{s} \rho\). Following the discussion in
		Section \ref{strong}, this should follow if the adjoint lifting for \(\GL_3\) in \cite{Gan} is
		proven to be strong.
	\end{remark}

	\section{Primitive solvable Artin representations}
	
	In this section, we give a definition of generalized tetrahedral and octahedral representations
	in arbitrary dimensions based on Suprunenko's work in the classification of primitive solvable
	subgroups of general linear groups in \cite[Chapter V]{Suprunenko}. We additionally deduce that
	all primitive solvable Artin representations of dimension 6 are automorphic using the
	Rankin-Selberg lift from \(\GL_2 \times \GL_3\) to \(\GL_6\) due to Kim and Shahidi in
	\cite{KS}.
	
	Let \(E^{\pm}_{p, m}\) denote the extraspecial group of order \(p^{1+2m}\), where the
	superscript \(+\) denote the extraspecial group of exponent \(p\) and \(-\) denote the one of
	exponent \(p^2\). The commutator map \([-,-]: E^{\pm}_{p, m} \to \mathbb{F}_p\) factors through
	\(E^\pm_{p, m}/Z(E^{\pm}_{p, m}) \cong \mathbb{F}_p^{2m}\) and defines a nondegenerate
	alternating bilinear form. Therefore, by Stone-von Neumann theorem and the theory of Weil
	representations, there exist \(p^m\)-dimensional primitive representations of
	\(E^\pm_{p, m}\rtimes\Aut(2m, p)\), where \(\Aut(2m, 2) \coloneq O^\pm(2m, 2)\) and
	\(\Aut(2m, p) \coloneq \Sp(2m, p)\) for \(p\) odd. If one take \(B
	\subset \Aut(2m, p)\) to be a solvable subgroup fixing no totally isotropic subspaces,
	then this gives a solvable primitive representation of dimension \(p^m\). It turns out that this
	classifies all solvable primitive representations up to central elements.

	\begin{theorem}[{\cite[Theorem 20.17, Theorem 20.18]{Suprunenko}}]
    Let \(\rho: \Gamma \hookrightarrow \GL_n(\mathbb{C})\) be a finite primitive solvable subgroup
		of \(\GL_n(\mathbb{C})\) and let \(n = p_1^{m_1}\ldots p_k^{m_k}\) be the prime factorization 
		of \(n\). Then
    \[
			\Gamma \cong Z(\Gamma) \circ \left((E^{\pm}_{p_1, m_1} \times \cdots \times E^{\pm}_{p_k,
			m_k}) \rtimes B\right)
    \]
    where \(Z(\Gamma)\) is the center of \(\Gamma\), the operation \(\circ\) denotes the central
		product, and \(B\) is a solvable subgroup of \(\prod_{i=1}^k \Aut(2m_i, p_i)\) acting
		naturally on the direct product of extraspecial groups.

		If \(B_i\) is the projection of \(B\) onto \(\Aut(2m_i, p_i)\), then each \(B_i\)
		fixes no totally isotropic subspaces. Moreover, there exists a character \(\chi: Z(\Gamma)\to
		\mathbb{C}^\times\) and primitive solvable representations
		\[
			\rho_i: E^{\pm}_{{p_i}, {m_i}}\rtimes B_i \hookrightarrow \GL_{p_i^{m_i}}(\mathbb{C})
		\]
		such that \(\rho\) is isomorphic to the restriction of
		\(\chi \otimes \rho_1 \otimes\cdots\otimes \rho_k\) to \(\Gamma\), which corresponds to the
		restriction from \(\prod_i B_i\) to \(B\).
  \end{theorem}
	
	\begin{remark}
		The approach of \cite{Suprunenko} is classical and explicitly work with matrices. Moreover, it
		phrases all results in the form of maximal primitive solvable subgroups, which are not finite
		because they are of the form \(\mathbb{C}^\times \circ (E^\pm_{p, m}\rtimes B)\) where \(B\)
		is also maximal. Generally speaking, these results follow from considering the Fitting
		subgroup of \(\Gamma\), i.e. the unique maximal normal nilpotent subgroup of \(\Gamma\).
		We note that the use of Fitting subgroup in the proof of cases of strong Artin conjecture has
		precedents in \cite{Zhang, Wong, Martin2}.
	\end{remark}

	It is straightforward to check that when \(n = 2, 3\), the description matches the explicit
	classification of finite primitive solvable subgroups. When \(n=4\), Martin carried out the
	classification in \cite{Martin2} and proved that only two open cases remain for the strong Artin
	conjecture. More precisely, when \(\rho\) is primitive solvable of symplectic type, the group
	\(B\) is either \(C_5, D_{10}\) or \(F_{20}\), the last of which denotes the Frobenius group of
	order 20. Martin proved that while strong Artin conjecture for the first case follows from
	exterior square lifting of \(\GL_4\) due to \cite{Kim}, the latter two cases will only follow
	from an additional functoriality for the base change of \(\GL_4\) over a non-normal quintic
	extension of number fields whose Galois closure has Galois group \(D_{10}\) and \(F_{20}\),
	respectively.

	We thus observe that the proof of strong Artin conjecture for the case \(p_i \mid \abs{B_i}\) 
	usually requires additional functoriality, commonly a non-normal base change, than the case
	\(p_i \nmid \abs{B_i}\). To this end, we propose a definition of tetrahedral and octahedral
	representations in arbitrary dimensions.

	\begin{definition}
		Let \(\rho: \Gamma_F \to \GL_n(\mathbb{C})\) be a primitive solvable Artin representation.
		We decompose \(\Im(\rho)\) in the above form and consider, for each prime \(p_i\mid n\), the
		subgroup \(B_i \subset \Aut(2m_i, p_i)\).
		We say \(\rho\) is (generalized) \textbf{tetrahedral} if \(p_i \nmid \abs{B_i}\) for every
		prime \(p_i\). We say \(\rho\) is (generalized) \textbf{octahedral} otherwise.
	\end{definition}

	\begin{remark}
		Lapid defines a class of representations of prime power dimensions which he calls tetrahedral
		in \cite[Section 5]{Lapid2}. When \(n\) is a prime power which is not prime, our definition is
		strictly more general. However, under the new definition it still remains true that the
		strong Artin conjecture for tetrahedral representations in arbitrary dimensions
		can still be proven, conditional on adjoint lifting from \(\GL_n\) to \(\GL_{n^2-1}\), using
		the same argument of Lapid.
	\end{remark}

	Suprunenko's classification also gives a natural strategy for proving strong Artin conjecture
	for primitive solvable representations of arbitrary dimensions. Namely, one first proves the
	modularity of \(\rho_i\) of prime power dimensions, such that \(\pi_i \leftrightarrow \rho_i\).
	One then invokes Rankin-Selberg lifts and twists by the character \(\chi\), from which
	one obtains \((\pi_1 \boxtimes \cdots \boxtimes \pi_k)\otimes \chi\). Finally, one applies base
	change from the extension defined by \(\prod B_i\) to that defined by \(B\). While all steps are
	highly nontrivial, we are interested in when Arthur-Clozel base change is enough for
	the last step. It motivates the following group theoretic lemma. Recall that a subgroup
	\(H \subset \Gamma\) is subnormal if the sequence of normalizers
	\(K_0 = H\), \(K_{i+1} = N_\Gamma(K_i)\) eventually terminates at \(\Gamma\). Equivalently, the
	sequence of commutators \(K_0 = \Gamma, K_{i+1} = [H, K_i]\) is eventually contained in \(H\).

	\begin{lemma}
		Let \(B_1, \ldots, B_k\) be finite groups and \(B \subset \prod_i B_i\) be a subdirect product,
		i.e. \(B\) surjects onto each \(B_i\). We identify each \(B_i\) naturally as a subgroup of \(\prod_i
		B_i\) and define \(N_i \coloneq B \cap B_i\). Then \(B\) is subnormal in \(\prod_i B_i\) if
		and only if \(B_i/N_i\) is nilpotent for all \(i\), in which case there exists a subnormal
		series from \(B\) to \(\prod_i B_i\) with each successive quotient being cyclic of prime order.
	\end{lemma}

	\begin{proof}
		(\(\Rightarrow\))
		The property of being subnormal is preserved and reflected after quotienting out a normal
		subgroup contained in \(B\). Thus, considering the quotient group \(\prod_i B_i / \prod_i
		N_i\), we may assume that each \(N_i\) is trivial. Then because \(X_0 = \prod_i B_i, X_{j+1} =
		[B, X_j]\) is eventually contained in \(B\), we have that \(Y_{i, 0} = B_i, Y_{i, j+1} = [B,
		Y_{i, j}]\) is eventually contained in \(B\). Since \(B_i\) is normal in \(\prod_i B_i\),
		\(Y_{i, j}\) is also eventually contained in \(B_i\), which means \(Y_{i, j}\) is eventually
		trivial. If we apply the projection map from \(B\) to \(B_i\), which is surjective, this means
		that the lower central series of \(B_i\) eventually terminates. Hence \(B_i\) is nilpotent.

		(\(\Leftarrow\))
		After quotienting by \(\prod_i N_i\), we have \(\prod_i B_i\) is nilpotent, so there is a
		subnormal series connecting from \(B\) to \(\prod_i B_i\) which may be refined to consist of
		steps of cyclic of prime degree.
	\end{proof}
	
	In the case \(n = 6\), we have the Rankin-Selberg lift from \(\GL_2\times\GL_3\) to \(\GL_6\)
	due to Kim and Shahidi as in \cite{KS} as well as enough conditions for applying
	Arthur-Clozel base change, so we have the following corollary.

	\begin{corollary}
		Let \(\rho: \Gamma_F \to \GL_6(\mathbb{C})\) be a primitive solvable Artin representation. Then
		the strong Artin conjecture is true for \(\rho\).
	\end{corollary}

	\begin{proof}
		All primitive solvable Artin representations of dimension 2 and 3 are known to be automorphic.
		For the base change step, we have \(p_1 = 2 \), \(p_2 = 3\), the subgroups
		\(B_1 \subset O^+(2, 2) \cong C_2\) or \(B_1 \subset O^-(2, 2) \cong S_3\), and
		\(B_2 \subset \Sp(2, 3) \cong \SL(2, 3)\). Moreover, since there are only two factors, by
		Goursat's lemma, \(B\) is the fibre product of a diagram \(B_1 \twoheadrightarrow Q
		\twoheadleftarrow B_2\). Hence, the previous lemma states that Arthur-Clozel base change
		suffices if \(Q\) is nilpotent. The only non-nilpotent choice of \(B_1\) is \(S_3\).
		However, no subgroup of \(\SL(2, 3)\) has \(S_3\) as a quotient.
	\end{proof}
	\printbibliography
\end{document}

%% file: preamble.tex
\documentclass[11pt, reqno]{amsart}

\usepackage{amsmath,amssymb} 
\usepackage{newpxmath,newpxtext} 
\usepackage{amsthm}   
\usepackage{thmtools} 
\usepackage{physics}  
\usepackage{geometry} 
\usepackage{setspace} 
\usepackage{framed}   
\usepackage{tikz-cd}   
\usepackage{tikz}
\usepackage[backend=biber, style=alphabetic, maxalphanames=10, sorting=nyt, maxbibnames=10]{biblatex} 
\usepackage{graphicx}
\usepackage{calc}
\usepackage{mathtools}

\newcommand{\Hom}{\operatorname{Hom}}

\newcommand{\GL}{\operatorname{GL}}
\newcommand{\SL}{\operatorname{SL}}

\newcommand{\PGL}{\operatorname{PGL}}
\newcommand{\PSL}{\operatorname{PSL}}
\newcommand{\Sp}{\operatorname{Sp}}

\newcommand{\Ad}{\operatorname{Ad}}
\newcommand{\BC}{\operatorname{BC}}
\newcommand{\AI}{\operatorname{AI}}
\newcommand{\Ind}{\operatorname{Ind}}
\newcommand{\Gal}{\operatorname{Gal}}

\newcommand{\Aut}{\operatorname{Aut}}

\newcommand{\Frob}{\operatorname{Frob}}
\newcommand{\rec}{\operatorname{rec}}

\makeatletter
\newcommand*{\bigboxplus}{%
  \DOTSB
  \mathop{\vphantom{\bigoplus}\mathpalette\matt@bigboxplus\relax}%
  \slimits@
}
\newcommand\matt@bigboxplus[2]{%
  \vcenter{\m@th\hbox{\resizebox{\widthof{$#1\bigoplus$}}{!}{$\boxplus$}}}%
}
\makeatother

\theoremstyle{plain}
\newtheorem{proposition}{Proposition}[section]
\newtheorem{theorem}[proposition]{Theorem}
\newtheorem{lemma}[proposition]{Lemma}
\newtheorem{corollary}[proposition]{Corollary}

\theoremstyle{definition}
\newtheorem{definition}[proposition]{Definition}

\theoremstyle{remark}
\newtheorem*{remark}{Remark}